\documentclass[a4paper, 11pt, onecolumn]{article}

\usepackage{xcolor}

\usepackage[utf8]{inputenc}
\usepackage{amsmath, amssymb}
\usepackage{amsfonts}
\usepackage{amsthm}
\usepackage{enumerate}
\usepackage{hyperref}

\usepackage[top=1in, bottom=1.3in, left=1in, right=1in]{geometry}

\newtheorem{theorem}{Theorem}[section]
\newtheorem{corollary}[theorem]{Corollary}
\newtheorem{lemma}[theorem]{Lemma}
\newtheorem{proposition}[theorem]{Proposition}

\theoremstyle{definition}
\newtheorem{remark}[theorem]{Remark}
\newtheorem{definition}[theorem]{Definition}
\newtheorem{example}[theorem]{Example}

\newcommand{\ds}{\displaystyle}
\newcommand{\C}{\mathbb{C}}
\newcommand{\R}{\mathbb{R}}
\newcommand{\Z}{\mathbb{Z}}
\newcommand{\Q}{\mathbb{Q}}
\newcommand{\N}{\mathbb{N}}
\newcommand{\M}{{\rm M}} 
\newcommand{\GL}{\rm GL}
\newcommand{\Qp}{\mathcal{Q}} 
\newcommand{\G}{\mathfrak G} 
\newcommand{\E}{\mathfrak E} 

\newcommand{\Abs}[1]{\left\lvert #1\right\rvert} 
\newcommand{\abs}[1]{| #1 |} 

\newcommand{\mfrac}[2]{\mbox{$\frac{#1}{#2}$}} 

\newcommand{\set}[1]{\left\{#1\right\}} 

\newcommand{\mat}[4]{\begin{pmatrix}#1 & #2\\#3 & #4\end{pmatrix}} 

\newcommand{\inv}{^{-1}}

\newcommand{\san}{\set{a_n}}
\newcommand{\szn}{\set{z_n}}
\newcommand{\spn}{\set{p_n}}
\newcommand{\sqn}{\set{q_n}}
\newcommand{\sdeln}{\set{\delta_n}}

\newcommand{\np}{_{n+1}}
\newcommand{\nm}{_{n-1}}

\newcommand{\an}{a_n}
\newcommand{\anp}{a\np}
\newcommand{\anm}{a\nm}

\newcommand{\zn}{z_n}
\newcommand{\znp}{z\np}
\newcommand{\znm}{z\nm}

\newcommand{\pn}{p_n}

\newcommand{\pnm}{p\nm}

\newcommand{\qn}{q_n}
\newcommand{\qnp}{q\np}
\newcommand{\qnm}{q\nm}

\newcommand{\rn}{r_n}

\newcommand{\rnm}{r\nm}

\newcommand{\deln}{\delta_n}
\newcommand{\delnm}{\delta\nm}

\newcommand{\Bb}{\overline{B}}
\newcommand{\zb}{\bar{z}}
\newcommand{\xib}{\bar{\xi}}
\newcommand{\etab}{\bar{\eta}}

\begin{document}
\date{} \author{S.G. Dani and Ojas Sahasrabudhe}
\title{{Generalized continued fraction expansions of complex numbers, and applications to quadratic and badly approximable numbers
}}
\maketitle

\section{Introduction}

Study of  continued fraction expansions for complex numbers was initiated by Adolf Hurwitz in~\cite{Hur-A}, where he considered continued fraction expansions in terms of Gaussian integers,  with respect to the nearest Gaussian integer algorithm (now referred to as the Hurwitz algorithm). It was also followed up, soon after, by Julius Hurwitz in \cite{Hur-J} who introduced  certain interesting variations. In the recent decades there have been extensions of the theme, involving expansions in terms of elements from other discrete subrings, expansions involving more general algorithms, and also expansions free of algorithmic considerations; the reader is referred to \cite{Lev}, \cite{Lak}, \cite{Hen}, \cite{DN}, \cite{D-gen}, \cite{D-lazy}, for various details in this respect. 

Apart from generalization of the framework and dealing with issues of convergence etc., one of the themes involved concerns extension of  
the classical theorem of Lagrange  characterizing  real quadratic surds as the numbers whose simple continued fraction expansions  are eventually periodic. The theme, involved in the works of  A.\,Hurwitz \cite{Hur-A} and J.\,Hurwitz \cite{Hur-J} establishing the analogue in their respective contexts, was taken up in a broader  framework in \cite{DN} and~\cite{D-gen}, where certain general sufficient conditions are described under which the analogue of Lagrange's theorem holds. 
    
In another direction, R.\,Hines  \cite{Hin} noted, more recently, the existence of circles in $\C$ all of whose points are badly approximable with respect to Gaussian integers, via a study of the continued fraction expansion of zeroes of Hermitian forms, with respect to the Hurwitz algorithm. 

In this paper we take up the study of these topics in the framework of the  generalized notion of continued fraction expansions as studied in \cite{DN} and \cite{D-gen};  and also 
introduce, in \S\,3,  the expansions with respect to all Euclidean subrings of $\C$, though the principal results focus on the rings of Gaussian and Eisenstein integers. Also, in \S\,3 we introduce  a common approach to studying zeroes of binary forms, via consideration of the action of $\GL(2,\Z)$ on the space of the forms (see Theorem~\ref{finmat}), putting the two streams mentioned above, concerning quadratic surds and badly approximable numbers respectively, in a common perspective. The proof of the theorem is facilitated by the notion of relative errors introduced here and certain simple observations about them (see \S\,2), which may be of independent interest. 
    
One of the crucial properties of the classical simple continued fraction expansions that the convergents are the ``best approximants" has found partial  generalization for continued fractions expansions associated with the nearest integer algorithms.  In \S\,4 we generalize the result, putting the issue involved in a broader perspective.    
    
For continued fraction expansions, understanding the behaviour of the sequence of denominators of the convergents plays a crucial role in many problems (see \S\,2 for the general form of these notions, as involved here). 
One of the questions addressed in the literature in this respect, starting from the paper of A. Hurwitz \cite{Hur-A}, concerns whether, and under what conditions, are the absolute values of the denominators monotonically increasing. In \S\,5 we discuss a general framework for studying the issue, developing further on \cite{DN} and \cite{D-gen}, and describe various situations, together with certain new examples, in which the desired monotonicity can be concluded. One of the features emphasized here is that there are vast classes of algorithms, and also non-algorithmic procedures, for which the monotonicity statement for the denominators holds for the corresponding continued fraction expansions, and in particular enables deducing analogues of Lagrange's theorem. We may note that as the standard Hurwitz algorithm does not share some of the interesting properties of  the classical simple continued fraction algorithm for real numbers, recourse to the flexibility and generality available can be helpful in some contexts; see~\cite{DN} for an example of this. In \S\,6 we specialize to the ring of Gaussian integers and describe a class of algorithms complementary to those studied in \cite{DN} and uphold monotonicity of the sequence of absolute values of the denominators. Theorem~\ref{Gaussian} proved in this respect generalizes the results of J.\,Hurwitz \cite{Hur-J} for continued fraction expansions with respect to even Gaussian integers, viz. $x+iy$ such that $x+y$ is even.

\section{Preliminaries}
Let $\Gamma$ be a discrete subring of $\C$ containing $1$. Let $z\in \C$.  We call a sequence $\san_{n\geq 0}$ of elements in $\Gamma$,
a \emph{continued fraction expansion of $z$} if  there exists a sequence $\szn_{n\geq 0}$ in $\C$ such that $z_0=z$,  $0<\abs{\zn-\an}<1$ and 
$\znp=(\zn-\an)\inv$ for all $n\geq 0$. 
We call $\san$ the sequence of \emph{partial quotients} of the continued fraction expansion, and $\szn_{n\geq 0}$ (which is unique) the corresponding \emph{iteration sequence}. We use the notation $z=\langle a_0, a_1,\ldots\rangle$ if 
$\san_{n\geq 0}$ defines a continued fraction expansion of $z.$

It may be noted that here we are considering only ``infinite continued fraction expansions", viz. $\san_{n\geq 0}$'s are infinite sequences. To ensure existence of such a sequence for $z$ (for suitable $\Gamma$'s; see Remark~\ref{rem1} below for more on this) we shall assume that $z$ is not contained in the quotient field of~$\Gamma$.  We denote by $K$ the quotient field of $\Gamma$ and by $\C'$ its complement in $\C$, namely $\C'=\C\backslash K$. To any sequence $\san_{n\geq 0}$ in $\Gamma$, we associate a pair of sequences 
$\spn_{n\geq -1}$ and $\sqn_{n\geq -1}$ defined by 
$$ p_{-1}=1, \ \ p_0=a_0, \ \ p_{n}=a_{n} p_{n-1} + p_{n-2}, \mbox{ and}$$
$$ q_{-1}=0, \ \ q_0=1, \ \ q_{n}=a_{n} q_{n-1} + q_{n-2},$$  
for $n= 1, 2, \dots $. 
We call $\spn_{n\geq -1}$, $\sqn_{n\geq -1}$ the {\it $\Qp$-pair}, and $\sqn_{n\geq 0}$  the \emph{denominator sequence}, corresponding to the sequence $\san$ of partial quotients. We also associate with the continued fraction expansion a sequence of matrices $\set{g_n}_{n\geq 0}$ defined by $$g_n:=\mat{\pn}{\pnm}{\qn}{\qnm} \hbox{ \rm  for all } n\geq 0.$$
where  $\spn_{n\geq -1}$, $\sqn_{n\geq -1}$  are as above. 
We recall that $\pn\qnm-\pnm\qn =(-1)^{n+1}$, for all $n\geq 0$, as may be proved inductively; thus, each $g_n$ has determinant $\pm 1$.

\subsection{General properties of continued fraction expansions}

We begin by noting certain interrelations between various sequences associated with a continued fraction expansion, following the notation as above. 

\begin{proposition}\label{lazy1}
	Let $\Gamma$ be a discrete subring in $\C$. Let $z\in \C'$, $\san_{n\geq 0} $ be a continued fraction expansion of $z$, $\szn_{n\geq 0}$ the corresponding iteration sequence and 
	{$\spn_{n\geq -1}$, $\sqn_{n\geq -1}$}  the corresponding $\Qp$-pair. Then  for all $n\geq 0$ the following statements hold: 
	\begin{enumerate}[(i)]
		\item \label{lazy1.1}
		$\qn z - \pn = (-1)^n (z_1\cdots \znp)\inv$; in particular 
		$\qn\neq 0$;
		\item \label{lazy1.3}
		$(\znp\qn+\qnm) z = \znp\pn+\pnm$; 
		\item \label{lazy1.4}
		$\abs{z-\frac{\pn}{\qn}} = \abs{\qn}^{-2}\abs{\znp +\frac{\qnm}{\qn}}\inv$. 
    \end{enumerate}
\end{proposition}

\begin{proof} 
	The first statement may be proved inductively, and the others can be deduced from it via simple manipulations. We omit the details; see \cite{D-gen} for an idea of the proofs. 
\end{proof}

\begin{proposition}\label{qntoinfinity} 
	Let the notation be as above. Then $|\qn|\to \infty$ as $n\to \infty$. 
\end{proposition}

\begin{proof}
	Suppose to the contrary that $\sqn$ has a  convergent subsequence; using Proposition~\ref{lazy1}\eqref{lazy1.1} we can then get that there exists an increasing sequence $\set{n_k}$ in $\N$ such that $\set{p_{n_k}}$ and $\set{q_{n_k}}$ are convergent. Since they are contained in the discrete subring $\Gamma$ they are eventually constant, viz. $p_{n_k}=p$ and $q_{n_k}=q$, with $p,q\in \Gamma$, for all large $k$. Then for sufficiently large $k$ we have 
	$$\abs{qz-p}=\abs{q_{n_{k+1}}z-p_{n_{k+1}}}= 
	\abs{z_1\cdots z_{n_{k+1}+1}}\inv<\abs{z_1\cdots z_{n_{k}+1}}\inv=\abs{q_{n_{k}}z-p_{n_{k}}}=\abs{qz-p},$$
	which is a contradiction. This proves the proposition. 
\end{proof}

\begin{remark}\label{rem1}
Existence of  continued fraction expansions over $\Gamma$, for $z\in \C'$, evidently involves the condition that $\Gamma$ should contain elements at distance less than $1$ from $z$, and the successive $z_n$'s as above. This condition holds for all $z$ precisely for $\Gamma$ which are Euclidean rings; see \cite{D-gen} for some details. The condition holds in particular for the ring $\G$ of Gaussian integers, generated by $1$ and $i$, and the ring $\E$ of Eisenstein integers, generated by $1$ and $\omega$, the latter being a nontrivial cubic root of $1$; these will be the subrings of interest in the examples discussed in later sections. 
\end{remark}

\begin{remark}
A standard procedure for generating continued fraction expansions is through algorithms. Let $\Gamma $ be a discrete subring which is Euclidean, and $K\subset \C$ its quotient field.  
By a $\Gamma$-valued \emph{continued fraction  algorithm}, we mean a map $f: \C\to \Gamma$ such that for all $z\in \C$, $\abs{z-f(z)}< 1$ (for convenience we adopt a slightly more restrictive notion here than in \cite{D-gen}, where equality is also allowed in place of the strict inequality). Given an algorithm $f$ and $z\in \C'$,  we get an iteration sequence $\szn$ by setting $z_0=z$,  and having defined $z_0, \dots, z_n$ for some $n\geq 0$, setting $\znp=(\zn-f(\zn))\inv$; we note that as $z\in \C'$, successively each $\zn$ may be seen to be in $\C'$, and hence $f(\zn)\neq \zn$. We thus get a(n infinite) continued fraction expansion,  in the sense as above,  when we choose $\an= f(\zn)$ for all $n\geq 0$. The reader is referred to \cite{DN}, \cite{D-gen}, and \S\,s~5 and~6 below, for a variety of examples of algorithms, when $\Gamma$ is either the ring of Gaussian integers or the ring of Eisenstein integers. 
\end{remark}

\subsection{Sequence of relative errors and neat subsets}

In this section we introduce certain notions concerning continued fraction expansions and discuss certain simple properties that play a crucial role in our results. 

As before let $\Gamma$ be a discrete subring of $\C$, $K$ the quotient field of $\Gamma$, $\C'=\C\backslash K$, $z\in \C'$ and let $\san$ be a continued fraction expansion of $z$ over $\Gamma$. Let $\szn $ be the associated iteration sequence, and $\spn_{n\geq -1}$, $\sqn_{n\geq -1}$ the corresponding $\Qp$-pair. 
 
The sequence $\deln$, $n\geq 0$ defined by 
$$\deln:=\qn(\qn z-p_n) =\qn^2\left(z-\dfrac{\pn}{\qn}\right)$$
is called the \emph{sequence of relative errors} corresponding to the expansion.

\begin{proposition}\label{theta1}
	For any  $n\geq 1$ such that $|\qnm| \leq |\qn|$ the following statements hold:
	
	i) $|\deln|\leq (|\znp|-1)\inv$ and 
	ii) $|\delnm| \leq |\deln|+1$.
\end{proposition}

\begin{proof}
	 i) This is a direct consequence of Proposition~\ref{lazy1}\eqref{lazy1.4} and the condition in the hypothesis. 
	
	ii) We have
	$\qn^2\delnm=\qn^2(\qnm^2 z-\pnm\qnm)=\qnm^2\deln+\qnm\qn (\pn\qnm-\pnm\qn).$ 
	By dividing by $q_n^2$ and recalling that $\pn\qnm-\pnm\qn=(-1)^{n+1}$ we get $$\delnm=\deln\left(\dfrac{\qnm}{\qn}\right)^2+(-1)^{n+1} \left(\dfrac{\qnm}{\qn}\right).$$  
	Since by assumption $|\qnm| \leq |\qn|$ it follows that $|\delnm| \leq |\deln|+1.$
\end{proof}

\begin{definition}
We say that a subset $N$ of $\N$ is \emph {neat} for the expansion if $$|\qnm|\leq |\qn| \hbox{ \rm for all } n\in N \ \ \hfill{ \rm and }\ \ \liminf_{n\in N}|\znp|>1. $$
\end{definition}

\begin{proposition}\label{theta2}
	If $N \subset \N$ is a neat subset for the continued fraction expansion then $\set{\deln \mid n\in N }$ is a bounded subset of $\C$.
\end{proposition}
\begin{proof}
	 We have 
	$\abs{\deln}=\abs{\qn}^2 \abs{z-\pn/\qn} =\abs{\znp+\qnm/\qn}\inv,$ 
	by  Proposition \ref{lazy1}\eqref{lazy1.4}, for all $n\in \N$. 
	For $n\in N$, since $\abs{\qnm/\qn}\leq 1<\abs{\znp}$, this implies  that 
	$\abs{\deln} \leq (\abs{\znp}-\abs{\qnm/\qn})\inv \leq (\abs{\znp}-1)\inv$.  
	Thus for any $\alpha \in (1, \liminf_{n\in N}\abs{\znp})$, we get  
	$\abs{\deln} \leq (\alpha -1)\inv$ for all sufficiently large $n\in N$.  
	Hence $\set{\deln \mid n\in N}$ is bounded.
\end{proof}

\begin{proposition}\label{neatsubsets}
	An infinite neat subset exists for a continued fraction expansion $\san$ if any one of the following conditions hold (in the notation as above):
	
	i) $\abs{\qnm} \leq \abs{\qn}$ for all $n\in \N$;
	
	ii) $\liminf_{n\in \N}\abs{\znp}>1$; 
	
	iii) the expansion is given by an algorithm $f$ for which there exists  $r\in (0,1)$ such that $\abs{\zeta - f(\zeta)}<r$ for all $\zeta\in \C$.  
	
	If the conditions as in (i) and (iii) hold then (the whole of) $\N$  is a neat subset.
\end{proposition}

\begin{proof}
	i) We recall that under the condition as in the hypothesis,  $\limsup_{n\in\N} \abs{\zn}>1$; see \cite{DN} or \cite{D-lazy}. Thus for any $\alpha \in (1, \limsup_{n\in\N} \abs{\zn})$, the set $ N=\set{n\in\N \mid \abs{\znp}>\alpha}$ is infinite, and it is a neat subset for the expansion. This proves~(i). 
	
	ii) By Proposition~\ref{qntoinfinity} $\abs{\qn}\to \infty$, and hence the set 
	$N=\set{n\in \N\mid \abs{\qnm}\leq\abs{\qn}}$ is infinite. The condition in the hypothesis then implies that $N$ is neat, proving~(ii). 
	
	iii) For any $n\geq 1$ we have $\abs{\znp}=\abs{\zn-\an}\inv=\abs{\zn-f(\zn)}\inv>r\inv>1$. Thus the condition as in (ii) holds in this case and hence the desired conclusion follows. 
	
	The last statement follows from the observation as in the proof of~(iii). 
\end{proof}

The whole of $\N$ being a neat set will be seen to have interesting consequences. The last part of the proposition therefore highlights the importance of knowing whether for an expansion the sequence $\set{\abs{\qn}}$ is monotonic. In \S\,s 5 and 6 we shall discuss conditions for this to hold, and give some new examples for which it holds.

\section{Zeroes of integral binary forms and quadratic polynomials}
 
Let $\Gamma$, $K$ and $\C'=\C\setminus K$ be as before.  
By $\C^2$ we shall denote the space of $2$-rowed column vectors, viewed also as $2\times 1$ matrices,  with entries in $\C$; for convenience we shall also write the elements as $(\xi,\eta)^t$ with $\xi, \eta \in \C$ (with the $t$ standing for transpose). 
By $\sigma :\C\to \C$ we denote the map which is either the identity map or the complex conjugation map. For any $\xi \in \C$ we denote $\sigma (\xi)$ by $\xi^\sigma$. 
 
We denote by $\M(2,\C)$ the space of all $2\times 2$ matrices with entries in $\C$.  For any discrete subring $\Gamma$ as above, and any $k\in \N$, we denote by $\M(2, \Gamma)$ and $\M(2, k\inv\Gamma)$ the subgroups consisting of elements $X$ whose entries are in $\Gamma$ and $k\inv \Gamma$ respectively. For any matrix $X$ we denote by $X^t$ the transpose of $X$.  For $X\in \M(2,\C)$ we denote by  $X^\sigma$ the matrix obtained by applying $\sigma$ to each entry of $X$.  We say that  $X\in \M(2, \C)$ is  {\it $\sigma$-symmetric} if $X^t=X^\sigma$; thus, $X$ is $\sigma$-symmetric if it is symmetric and $\sigma$ is the identity map or Hermitian symmetric and $\sigma$ is the complex conjugation map. 

For $X= \mat{A}{B}{C}{D}\in \M(2,\C)$, by the {\it $\sigma$-form} corresponding to $X$ we mean the function $f:\C \times \C\to \C$ defined by,  for all $ \xi, \eta \in \C$,
 $$f(\xi, \eta)=A \xi^{\sigma}\xi +B\xi^{\sigma}\eta +C\eta^{\sigma}\xi +D\eta^\sigma \eta.$$ 
We note that $f(\xi, \eta)$ is the entry of the $1\times 1$ matrix $(\xi, \eta)^\sigma X (\xi, \eta)^t$.   
By a {\it nontrivial zero} of the $\sigma$-form $f$ we mean $(\xi,\eta)^t \in \C^2$ such that $\eta \neq 0$, $\xi/\eta \in \C'$ and $f(\xi, \eta)=0$.

\subsection{A theorem for $\sigma$-symmetric forms}

For any $g\in \GL(2,\C)$ and $X\in \M(2,\C)$ let $X_{g,\sigma}=
(g^t)^\sigma Xg$. We prove the following:

\begin{theorem}\label{finmat}
 Let $X\in\M(2,\C)$ be a $\sigma$-symmetric matrix, $f$ the corresponding $\sigma$-form and  let $(\xi, \eta)^t\in \C^2$ be a nontrivial zero of $f$. 
 Let $\set{g_n}$ be the sequence in $\GL(2,\Z)$ associated with a continued fraction expansion of $\xi/\eta$ and $N\subset\N$ be a neat infinite subset for the expansion. Then $\set{X_{g_n,\sigma} \mid n\in N}$ is a bounded subset in $\M(2,\C)$. Moreover, for any $X\in \M(2, k\inv \Gamma)$, where $k\in \N$,  $\set{X_{g_n,\sigma} \mid n\in N}$ is finite. 
\end{theorem}

\begin{proof}
 Let $X= \mat{A}{B}{C}{D}\in \M(2,\C)$ be given; thus for any $\xi, \eta \in \C$ we have $f(\xi,\eta) =A \xi^{\sigma}\xi +B\xi^{\sigma}\eta +C\eta^{\sigma}\xi +D\eta^\sigma \eta$. For any $g\in \GL(2,\C)$, let $f^g:\C^2\to \C$ denote the function defined by $f^g:=f(g(\xi,\eta)^t)$.
 Then we see  that 
 $$f^{g_n}(\xi, \eta)= A_n \xi^{\sigma}\xi + B_n\xi^{\sigma}\eta + C_n\eta^{\sigma}\xi + D_n\eta^\sigma \eta$$ for all $\xi, \eta$, where $A_n, B_n, C_n, D_n$ are the entries of $X_{g_n,\sigma}$, namely, $X_{g_n,\sigma}=\mat{A_n}{B_n}{C_n}{D_n}$. 
 \\Substituting for $g_n$ as $\mat{\pn}{\pnm}{\qn}{\qnm}$ (see \S\,2) we see that 
 $$A_n=f(\pn,\qn) \hbox{ \rm and } D_n=f(\pnm,\qnm)  \hbox{ \rm for all }n\geq 0.$$
 Now let $z=\xi/\eta$ and $\deln=\qn^2(z-\frac{\pn}{\qn})$ be the sequence of relative errors for the continued fraction expansion as in the hypothesis. Thus we have 
 $\frac {\pn}{\qn}=z-\frac{\deln}{\qn^2}$ for all $n\geq 0$. Hence 
 $$A_n=f(\pn,\qn)=\qn^2\,f\left(\frac{\pn}{\qn}, 1\right)=\qn^2\,f\left(z-\frac{\deln}{\qn^2}, 1\right).$$
 On substituting in the expression for $f$, the values $\xi=z-\frac{\deln}{\qn^2}$ and $\eta =1$, the preceding equation yields
 $$\qn^{-2}A_n= A\left(z-\frac{\deln}{\qn^2}\right)^\sigma \left(z-\frac{\deln}{\qn^2}\right)+B\left(z-\frac{\deln}{\qn^2}\right)^\sigma +C\left(z-\frac{\deln}{\qn^2}\right)+D.$$ 
 As  $f(z,1)=f(\xi,\eta)=0$, it follows that 
 $$\qn^{-2}A_n= A\left((\gamma_n^\sigma+\gamma_n)z+\gamma_n^\sigma\gamma_n\right)
 + B\gamma_n^\sigma + C\gamma_n,$$ where $\gamma_n=-\deln/\qn^2$.
 We have $\abs{\gamma_n}=\abs{\deln}/\abs{\qn}^2$ and $\abs{\gamma_n^{\sigma}\gamma_n}\leq \abs{\deln}^2/\abs{\qn}^2$, as $\abs{\qn}\geq 1$. 
 Thus we get that 
 $$|A_n| \leq \left(|2Az|+|B|+|C|\right)|\deln|+|A|\,|\deln|^2, \hbox{ \rm for all } n\geq 0.$$
 Also, $|D_n|= |f(\pnm,\qnm)| \leq (|2Az|+|B|+|C|)\,|\delnm|+|Az|\,|\delnm|^2$, for all  $n\geq 0$.
 
 Now let $N$ be a neat subset of $\N$ with respect to the continued fraction expansion as in the hypothesis. Then there exists a constant $M>0$ such that for all $n\in N$, $|\deln|\leq M$, and by Proposition~\ref{theta1} we also have $|\delnm|\leq |\deln|+1 \leq M+1$. 
 Thus the above conclusions imply that $\set{A_n\mid n\in N}$ and $\set{D_n\mid n\in N}$ are bounded. 
 
 Now recall that by hypothesis $X$ is $\sigma$-symmetric. Therefore $X_{g_n,\sigma}$ is also $\sigma$-symmetric, and in particular $C_n= B_n^\sigma$ for all $n$. Thus the determinants of $X$ and $X_{g_n,\sigma}$ are, respectively, $AD-BB^\sigma $ and $A_n D_n- B_n B_n^\sigma$, for all $n$. From the definition of $X_{g_n,\sigma}$ we see that $\det(X_{g_n,\sigma})=\det (X)$ for all $n$. Hence we get that $B_nB_n^\sigma =A_nD_n-AD +BB^\sigma $. 
 Since $\set{A_n \mid n\in N}$ and $\set{D_n \mid n\in N}$ are bounded, the preceding conclusion shows that for either choice of $\sigma$, the sequence $\set{|B_n|^2\mid n\in N}$ is bounded. Thus $\set{B_n\mid n\in N}$ is  bounded, and since $C_n= B_n^\sigma$ for all $n$, it follows also that $\set{C_n\mid n\in N}$ is bounded. This completes the proof of the first assertion in the theorem. The second assertion is immediate from the first and the discreteness of $k\inv\Gamma$.  
\end{proof}

\begin{remark}\label{thebound}
A perusal of the proof of Theorem~\ref{finmat} shows, in the notation of the theorem, that  all entries of $X_{g_n,\sigma}$, $n\in N$, are 
bounded by a computable constant depending only on $X$ and $\sup\,\set{|\deln|\mid n\in N}.$  
\end{remark}

We shall next apply  the theorem to study zeroes of $\sigma$-forms. Towards this we first note here the following, in the notation as above. 

\begin{lemma}\label{lem}
 Let $X\in\M(2,\C)$ and let $f$ be the corresponding $\sigma$-form. Let $z\in \C'$, $\szn$ be the iteration sequence of a continued fraction expansion of $z$ and let $\set{g_n}$  be the associated sequence of matrices in $\GL(2,\C)$. 
 For $n\in \N$, let $f_n$ be the $\sigma $-form corresponding to $X_{g_n,\sigma}$. 
 Then  $f(z,1)=0$ if and only if $f_n(\znp,1)=0$.  
\end{lemma}

\begin{proof}
 The definition of the $\sigma$-form $f$ corresponding to $X$ readily shows that $(\xi, \eta)^t$ is a zero of $f$ if and only if, for any $g\in \GL(2,\C)$, $g(\xi, \eta)^t$ is a zero of the $\sigma$-form corresponding to ${X_{g,\sigma}}$. On the other hand, for all $n\in \N$, $g_n(\znp, 1)^t= (\pn\znp+\pnm, \qn\znp+\qnm)$, and by Proposition~\ref{lazy1}\eqref{lazy1.3} it is a non-zero multiple of $(z,1)^t$. Hence $(z,1)$ 
 is a zero of $f$ if and only if $(\znp,1)$ is a zero of $f_n$. 
\end{proof}

\subsection{Application to roots of quadratic polynomials} 

Now let $\Gamma$ be a discrete subring of $\C$ and $K$ be its quotient field. Let  $f(z)=az^2+bz+c$ be a quadratic polynomial, with $a,b, c\in \Gamma$ which is irreducible over $K$. It may be recalled that in the case of a quadratic polynomials with coefficients in $\Z$ 
which is irreducible over $\Q$ the classical Lagrange theorem asserts that for any root of the polynomial the partial quotients $\san$ in the simple continued fraction expansion are eventually periodic, namely there exist $k$ and $n_0$ such that $a_{n+k}=a_n$ for all $n\geq n_0$. Analogous results have been known in the framework as above, for complex quadratic polynomials, under various conditions on the continued fraction expansions; see \S 5 for some details. In the following Corollary we establish a version of such a result for a larger class of continued fraction expansions, namely those admitting a neat subset. This puts the issue in a general perspective and yields the previously known cases as well as certain new examples, which we shall discuss in the next section. 

It is easy to see that if $\san$ is the continued fraction expansion of a number $z\in \C'$ and it is eventually periodic then $z$ is a quadratic surd over $K$, namely the root of an irreducible quadratic polynomial. The import of various versions of Lagrange's theorem concerns the converse. In this regard we prove the following. 

\begin{corollary}\label{Lagrange}
 Let $z$ be the root of a quadratic polynomial $az^2+bz+c$, with $a,b,c \in \Gamma$, which is irreducible over $K$. 
 Let $\san_{n\geq 0}$ be a continued fraction expansion of $z$ with $\szn_{n\geq 0}$ as the associated iteration sequence. Suppose there exists an infinite neat subset for the expansion. Then  $\set{\znp\mid n\in N}$ is finite. 
 If moreover the continued fraction expansion is associated with an algorithm, then $\san$
 is eventually periodic. 
\end{corollary}

\begin{proof}
 We shall apply Theorem~\ref{finmat} with $X= \mat{a}{\frac{1}{2} b}{\frac{1}{2} b}{c}$, where $a,b,c$ are the coefficients of the polynomial as in the statement,  and $\sigma $ the identity map. Then the corresponding $\sigma$-form is given by 
 $f(\xi,\eta)=a\xi^2+b\xi\eta+c\eta^2$ for all $\xi, \eta \in \C$.  As the polynomial $az^2+bz+c$ is  irreducible over $K$, it follows that $z\in \C'$. 
 Thus $(z,1)$ is a nontrivial zero of $f$, and the theorem implies, in the notation as before, that $\set{{X_{g_n, Id}} \mid n\in N}$ is a finite set. Let $f_n$ denote the form corresponding to $X_{g_n, Id}.$ Then we get that $\set{f_n(\zeta, 1) \mid n\in N}$ is a finite collection of polynomials in $\zeta$. By Lemma \ref{lem} each $\znp$, $n\in N$, is a root of one of the polynomials from this finite collection. It follows that 
 $\set{\znp \mid n\in N}$ is finite, which proves the first assertion in the corollary. 

 Now suppose that the continued fraction expansion is with respect to an algorithm. As $\set{\znp \mid n\in N}$ is proved to be finite there exist $n_0$ and $k$ in $\N$ such that $z_{n_0+k}=z_{n_0}$. Since the continued fraction expansion follows an algorithm, this implies successively that $a_{n_0+k}=a_{n_0}$, $z_{n_0+1+k}=z_{n_0+1}, \dots$, and hence $a_{n+k}=a_n$ for all $n\geq n_0$, which shows that $\san$ is eventually periodic. 
\end{proof}
 
Corollary \ref{Lagrange} together with Proposition~\ref{neatsubsets} yields the following:

\begin{corollary}\label{Lag-mon}
 Let $z$ be the root of a quadratic polynomial $az^2+bz+c$, with $a,b,c \in \Gamma$, which is irreducible over $K$. 
 Let  $\san_{n\geq 0}$ be the continued fraction expansion of $z$ corresponding to an algorithm $f$. If either $|\qnm|\leq |\qn|$ for all $n\geq 1$, or 
 $\set{z-f(z) \mid z\in \C'}$ is contained in $B(0,r)$ for some $r<1$, then $\san_{n\geq 0}$ is eventually periodic.
\end{corollary}

It may be noted that the condition of $\set{z-f(z) \mid z\in \C'}$ being contained in $B(0,r)$ for some $r<1$ holds for the nearest integer algorithms and also various other algorithms considered in \cite{DN} and \cite{D-gen}, and hence the above corollary applies in these cases. In the next section we describe certain other algorithms for which $\set{z-f(z) \mid z\in \C'}$ is not contained in $B(0,r)$ for any $r<1$, but the monotonicity condition on the denominator sequence as in the corollary holds.

\subsection{Application to roots of Hermitian quadratic polynomials}

We next apply Theorem \ref{finmat} and prove the following analogue for hermitian quadratic polynomials, viz. functions of the form $H(z,1)$, where $H$ is a Hermitian binary form, showing that if $z\in \C'$ is a root of $H(z,1)$ then, under certain general conditions, the partial quotients $\san$ are bounded; this extends a result of Hines \cite{Hin}, where it is proved in the special case of the nearest integer algorithm. 
We begin by noting the following.

\begin{remark}\label{rem:zero} 
Let $P(z)=az\zb+b\zb+\bar{b}z+c$ where $a,b,c\in \Gamma$ and $a,c\in \R$ be a Hermitian quadratic polynomial. If $a\neq 0$ then the set of roots of $P$ consists of the circle $\abs{z+\frac{b}{a}}^2=(|b|^2-ac)/a$ if the latter is nonnegative (a degenerate circle with a  single point if it is zero), and is empty if it is negative. If $a=0$ and $b\neq 0$ then the set of roots is an affine line in $\C$, containing a dense set of rational points.
\end{remark}

\begin{corollary}\label{hbounded}
 Let  $P(\zeta)=a\zeta \bar{\zeta}+b\bar{\zeta}+\bar{b}\zeta +c$, where $a,b,c\in \Gamma$ with $a, c\in \R$ be a Hermitian quadratic polynomial. Let  $z\in \C'$ be a root of $P$, $\san_{n\geq 0}$ be a continued fraction expansion of $z$, $\szn_{n\geq 0}$ the associated iteration sequence, and $\spn, \sqn$ be the corresponding $\Qp$-pair. 
 Suppose that $\abs{\qn}\leq\abs{\qnp}$ for all $n$. Then at least one of the following holds:
 
 i) there exists an increasing sequence $\set{n_k}$ in $\N$ such that $P(p_{n_k}/q_{n_k})=0$ for all $k$;

 ii) $\san$ is bounded. 

 \noindent In particular if $P$ has no root in $K$, then $\san$ is bounded, and moreover, if $\sup_{n\in \N} |\deln|=M <\infty$, a bound for $ \san$ can be given depending only on $a,b,c$ and $M$. 
\end{corollary}

\begin{proof}
 Let $X= \mat{a}{b}{\bar{b}}{c}$, where $a,b,c$ are the coefficients of $P$ as in the hypothesis. Let $\sigma$ denote the complex conjugation map. Then $X$ is $\sigma$-symmetric and the corresponding form is given by $f(\xi,\eta)=a\xib \xi+b\xib \eta+\bar b\etab \xi+c\etab \eta$, for all $\xi, \eta \in \C$.  
 Thus  $(z,1)$ is a zero of $f$. Let $\set{g_n}$ be the sequence of matrices associated with the continued fraction expansion as in the hypothesis. 
 Consider any neat infinite subset $N$ for the expansion; by Proposition~\ref{neatsubsets}\,(i) such subsets exists.  Then, in the notation as before, by Theorem \ref{finmat} we get  that $\set{{X_{g_n, \sigma}} \mid n\in N}$ is a finite collection of matrices. For $n\in \N$, let $f_n$ denote the $\sigma$-form corresponding to $X_{g_n, \sigma}$ and let $\varphi_n$ be the Hermitian polynomial defined by $\varphi_n(\zeta)=f_n(\zeta ,1)$ for all $\zeta \in \C$. 
 Then $\set{\varphi_n\mid n\in N}$ is a finite collection of polynomials in $\zeta$, and  by Lemma \ref{lem} each $\znp$, $n\in N$, is a root of $\varphi_n$. We shall use this observation to prove the assertion in the Corollary. 

 Suppose that statement (i) does not hold. Then there exists $n'\in \N$ such that for $n\geq n'$, $P(\pn/\qn)\neq 0$. We fix $\lambda >1$ and put $N=\set{n\in\N \mid n\geq n' \hbox{ \rm and }|\znp|>\lambda}$. Since by assumption $\set{|\qn|}$ is monotonic, it follows that $N$ is a neat subset for the expansion, and in the notation as above $\set{f_n}_{n\in N}$ is a finite collection. 
 As recalled in the proof of Theorem~\ref{finmat}, in the expression for $f_n$ in terms of $\zeta$ and $\bar{\zeta}$, the coefficient of $z\zb$ is $f(\pn,\qn)$. Hence for $n\geq n'$ the leading coefficient of $\varphi_n$ is nonzero and therefore by Remark~\ref{rem:zero}, the set of roots of $\varphi_n$, if nonempty, consists of a circle. Since any $\znp$, $n\in N$ is a root of $\varphi_n$ and $\set{\varphi_n\mid n\in N}$ is a finite collection, it follows that $\{\znp\mid n\in N\}$ is contained in a union of finitely many circles in $\C$ and hence form a bounded subset of $\C$. On the other hand, for $\znp$, for $n\notin N$, by the definition of $N$ we have either $n\leq n'$ or $|\znp|\leq \lambda$. Altogether, we get that $\set{\zn\mid n\in \N}$ is bounded. Since for all $n$, $|\an|\leq |\zn|+1$ it follows that 
 $\san$ is bounded in $\C$,  which  shows that  statement\,(ii) holds; this proves the first assertion in the Corollary. 

 Finally suppose that $\sup_{n\in \N}|\deln| =M<\infty$. We recall that for $n\in N$, $\znp$ is contained in the circle defined by $\varphi_n(\zeta)=0$, and hence $|\znp|$ admits a bound depending only on the coefficients of $f_n$, namely the entries of $X_{g_n, \sigma}$. Hence by Remark~\ref{thebound}, $\set{\znp \mid n\in N}$ admits a bound depending only on $a,b,c$ and $M$, and hence so does $\set{\zn \mid n\geq 0}$, as $\lambda $ is an absolute constant. Since $|\an|\leq |\zn|+1$ for all $n$, this proves the last statement in the Corollary. 
\end{proof}

Corollary~\ref{hbounded} generalizes a result of Hines \cite{Hin} proved in the special case when $\Gamma$ is the ring of Gaussian integers, the continued fraction expansion is with respect to the nearest integer algorithm, and $P$ has no root in $K$. The partial quotients in the continued fraction expansion of a $z\in \C'$ being bounded relates to $z$ being ``badly approximable" by elements of $K$. This connection was discussed in \cite{Hin} in the case of expansions with respect to the nearest integer algorithm, for the ring  of Gaussian integers. In the next section we discuss the approximability issues in the broader context of the more general continued fraction expansions as considered in the earlier sections here, and relate the property of the partial quotients being bounded to bad approximability, under certain general conditions on the expansions.

\section{Approximation properties of convergents}

It is well-known that for the classical simple continued fraction expansions of real numbers, the convergents are the ``best approximants", in the sense that if $\pn/\qn$ is a convergent in the expansion of $x\in \R$ then for any $1\leq q\leq \qn$, $\abs{x-\frac{\pn}{\qn}}\leq \abs{x-\frac{p}{q}}$ (with equality only for $q=\qn$). Thus the ``quality" of approximation (namely how well they approximate the number in question compared to other rationals) is an issue in the study of continued fraction expansions. No strict equivalent of the above is known for any algorithmic expansions of complex numbers (see however \cite{Lev} for exploration on this question in another direction). There are however weaker versions of the result, proving the analogous relation upto a constant multiple (see \cite{Hin} and other references cited there). In this respect it turns out to be more fruitful to compare $\abs{qz-p}$ with $\abs{\qn z-\pn}$, where $p,q\in \Gamma$ (the latter being a Euclidean subring of $\C$ as before), $q \neq 0$ and an appropriate convergent of $z$, in place of $\abs{z-\frac{p}{q}}$ and $\abs{z-\frac{\pn}{\qn}}$; we note that in the setting as above, if $\abs{qz-p} \geq c\,\abs{\qn z-\pn}$, for some $c>0$,  then $$\abs{z-\frac{p}{q}}=\abs{q}\inv\abs{qz-p}\geq c\,\abs{q}\inv\abs{\qn z-\pn}\geq c\,\abs{\qn}\inv\abs{\qn z-\pn}=c\,\abs{z-\frac{\pn}{\qn}}.$$

\subsection{A general result on approximations}

The following proposition generalizes the result on comparison of $|qz-p|$ and $|\qn z-\pn|$ known earlier (see \cite{Hin}) in the case of continued fraction expansions corresponding to the nearest integer algorithms; along the way the constant factor involved is also clarified. 

Let $\Gamma$ be a Euclidean subring of $\C$, $K$ be the quotient field of $\Gamma$, and $\C'=\C\backslash K$. We denote by $\nu (\Gamma)$ the number $\min \set{|\gamma| \mid \gamma \in \Gamma, |\gamma|>1}$; as $\Gamma$ is discrete, it follows that $\nu (\Gamma)>1$. 

\begin{proposition}\label{app-pr} 
 Let $z\in \C'$ and $\san$ be a continued fraction expansion of $z$, $\szn$ be the corresponding iteration sequence, $\spn,\sqn$ the corresponding $\Qp$-pair and $\sdeln$ be the sequence of relative errors. Let $n\in \N$ be such that $|\deln|<\nu (\Gamma)$ and let 
 $\theta_n= \max\set{\frac{|\deln|}{\nu (\Gamma)},|\znp|\inv}$. 
 Then for all $q\in \Gamma$ such that $|\qnm|< |q| \leq \abs{\qn}$ and all $p\in \Gamma$, we have $|qz-p| \geq (1-\theta_n) |\qn z-\pn|.$
\end{proposition} 

\begin{proof}
 Let $p,q\in \Gamma$, $q\neq 0$, such that $|\qnm|<|q|\leq\abs{\qn}$ be given. Since $\pn\qnm-\pnm\qn= (-1)^{n+1}=\pm 1$, there exist (uniquely determined) $\varphi , \psi \in \Gamma$ such that $p=\varphi\pn+\psi\pnm$ and $q=\varphi\qn+\psi\qnm$. We note that since $|\qnm|<|q|$, $\varphi \neq 0$, and that  if $\psi =0$ then $p=\pn$ and $q=\qn$, so the desired assertion holds. We may therefore assume that $|\varphi|, |\psi|\geq 1$. 

 By inverting the defining (linear) relations for $\varphi, \psi$ we see that $\psi = (-1)^{n+1}(\pn q-\qn p)$, so $\abs{\frac{\pn}{\qn}-\frac{p}{q}}=\abs{\frac{\psi}{q\qn}}.$ Hence we have 
 $$\abs{qz-p}=|q|\Abs{\left(\frac{\pn}{\qn}-\frac{p}{q}\right)+ \left(z-\frac{\pn}{\qn}\right)}
 \geq |q|\Abs{|\frac{\psi}{q\qn}|-\abs{z-\frac{\pn}{\qn}}} 
 = \Abs{|\frac{\psi}{q\qn}|-|\frac{q}{\qn}|\,|\qn z-\pn|}.$$
 Recalling that $\deln=\qn(\qn z-\pn)$ we have $|\qn\inv\psi|=|\psi\deln\inv||\qn z-\pn|.$  
 Then, as $|q| \leq |\qn|$, the preceding observation shows that 
 $$|qz-p|\geq (|\psi| |\deln|\inv -1)|\qn z-\pn|;$$ we recall here that by hypothesis 
 $|\psi| |\deln|\inv >1$. Now suppose first that $|\psi|>1$. Then $|\psi|\geq \nu (\Gamma) $, and we have  
 $|qz-p|\geq (\nu (\Gamma) |\deln|\inv-1) |\qn z-\pn|$. Noting that for any $t>1$, $t-1>1-t\inv$, we have $\nu(\Gamma) |\deln|\inv -1>1-\nu (\Gamma)\inv |\deln|$ and hence 
 $$|qz-p|\geq \left(1-\frac{|\deln|}{\nu (\Gamma)}\right) |\qn z-\pn|\geq (1-\theta_n)|\qn z-\pn|,$$ 
 as sought to be proved. 
 Now suppose $|\psi |=1$. Then substituting for $p$ and $q$ we get 
 $$qz-p= (\varphi\qn+\psi\pn)z-(\varphi\qnm+\psi\pnm)= \varphi (\qn z-\pn) + 
 \psi (\qnm z-\pnm).$$  Hence, using Proposition~\ref{lazy1}\eqref{lazy1.1} we get that 
 $$\Abs{\frac{qz-p}{\qn z-\pn}}= \Abs{\varphi+ \psi\,\frac{(-1)^{n-1}z_1 \cdots \zn}{(-1)^{n}z_1 \cdots \znp}}=\abs{\varphi -\psi \znp\inv}\geq 1-\theta_n,$$ 
 since $|\varphi|\geq 1$, $|\psi|=1$, and  $|\znp\inv| \leq \theta_n$.  Thus $|qz-p|\geq (1-\theta_n) |\qn z-\pn|$, in this case as well. 
\end{proof}

\begin{remark}\label{rem:nearest}
In the case of the nearest integer algorithms on the Euclidean subrings the values of the infima of $\inf |\deln|\inv$ over the expansions of all complex numbers have been determined in \cite{Lak}, and using them it is noted in \cite{Hin} that $\sup |\deln|$ are bounded uniformly, independent of the point $z\in \C'$, by a constant less than $\nu (\Gamma)$. It is also easy to see for these expansions that there exists a $c>1$, independent of $z\in \C'$,
such that $\inf |\zn|>c$ for all $n$. Hence in these cases the conclusion as in Proposition~\ref{app-pr} holds with fixed positive constant in place of $1-\theta_n$, independent of both $n$ and the point $z$.   
\end{remark}

\begin{remark}\label{rem:badapp}
In the setting of Corollary~\ref{badapp}, for $n$ such that $|\znp|>1+\nu(\Gamma)\inv$ and $|\qnm|\leq |\qn|$  we automatically have (see Proposition~\ref{theta1}\,(i)) 
$|\deln| =\abs{\znp+\frac{\qnm}{\qn}}\inv <\nu (\Gamma)$.
Expansions satisfying the stronger condition uniformly for all $n$, namely such that $\inf_{n\in \N}|\znp|>1+\nu(\Gamma)\inv$, would not be possible for a general $z\in \C'$ when $\nu (\Gamma)=\sqrt{2}$, as can be easily verified. On the other hand when $\nu (\Gamma)=\sqrt{3}$, it is possible to find expansions with the property, for all $z\in \C'$; in particular the condition holds for expansions with respect to the nearest integer algorithms of the Eisenstein ring $\E$ and the ring $\Z[\frac{1}{2}(1+\sqrt{11})]$. The continued fraction expansions corresponding to the class of $\E$-valued algorithms described in Example~\ref{Eisen-exampl}, in the next section, also have this property, and hence conclusion as in Proposition~\ref{app-pr} holds for the corresponding expansions, with a uniform choice for $\theta_n$. 
\end{remark}

\subsection{Badly approximable numbers}

We recall that a complex number $z$ is said to be {\it badly approximable}, with respect to a discrete subring $\Gamma$, if there exists a $\delta >0$ such that $|z-\frac{p}{q}|\geq \delta /|q^2|$, for all $p, q\in \Gamma$, $q\neq 0$. Clearly if $z$ is badly approximable then $z\notin K$. Consider a badly approximable number $z$ and let 
$\san$ be a continued fraction expansion of $z$, $\szn$ the corresponding iteration sequence and $\spn,\sqn$ be the corresponding $\Qp$ pair, and suppose that $|\qn/\qnp|\leq 1$ for all $n$. Then by Proposition~\ref{lazy1}\eqref{lazy1.4} we have $\abs{\znp +\frac{\qnm}{\qn}}\inv =\abs{\qn}^2\abs{z-\frac{\pn}{\qn}}\geq \delta$, for some $\delta>0$, for all $n$, as $z$ is badly approximable. Hence for all $n\geq 0$ we have $|\znp |\leq |\frac{\qnm}{\qn}|+ \delta\inv \leq 1+\delta\inv$, and $|\anp| \leq |\znp|+1\leq 2+\delta\inv$. Thus we see that for a badly approximable $z$ and any continued fraction expansion $\san$ such that the corresponding sequence $\set{\abs{\qn}}$ is monotonic, the partial quotients $|\an|$ are bounded. The converse statement, that if the partial quotients are bounded then the number is badly approximable, is known to hold in the case of the expansions corresponding to the nearest integer algorithms on the Euclidean rings (see \cite{Hin}).   

The following Corollary generalizes this property. Recall that for a Euclidean subring $\Gamma$,  $\nu (\Gamma )$ denotes $\inf \set{|\gamma| \mid \gamma\in\Gamma, |\gamma|>1}$; 
thus $\nu (\Gamma)=\sqrt{2}$ when $\Gamma$ is $\G=\Z [i], \Z [i\sqrt{2}]$ or 
$\Z [\frac{1}{2}(1+i\sqrt{7})]$ and $\sqrt{3}$ for $\Z[\frac{1}{2} (1+i\sqrt{3})]$ or 
$\Z[\frac{1}{2} (1+i\sqrt{11})]$.

\begin{corollary}\label{badapp} 
 Let $\Gamma$ be a Euclidean subring of $\C$ and $K$ be the quotient field of $\Gamma$. Let $z\in \C'=\C\backslash K$ and $\san$ be a continued fraction expansion of $z$, 
 such that the corresponding sequence $\set{\abs{\qn}}$ is monotonic. Let $\szn$ be the corresponding iteration sequence and $\sdeln$ be the sequence of relative errors. Suppose that $\inf |\zn|>1$, and that ($\sdeln$ is bounded and)  $\limsup |\deln| < 
 \nu (\Gamma)$. Then $z$ is badly approximable with respect to $\Gamma$ if and only if $\san$ is bounded. 
\end{corollary} 

\begin{proof}
 We have already observed that if $z$ is badly approximable then $\san$ is bounded. Now suppose that $\san$ is bounded. As before let $\szn$ denote the corresponding iteration sequence and $\spn, \sqn$ the corresponding $\Qp$-pair. By Proposition~\ref{lazy1}\eqref{lazy1.4} for all $n$ we have $\abs{z-\frac{\pn}{\qn}} = \abs{\qn}^{-2}\abs{\znp +\frac{\qnm}{\qn}}\inv$. Since by hypothesis  $|\qnm/\qn|\leq 1$, we get that $\abs{\znp +\frac{\qnm}{\qn}}\leq \abs{\znp}+1 \leq \abs{\anp}+2$, and hence $\abs{z-\frac{\pn}{\qn}}\leq \abs{\qn}^{-2}\abs{\anp+2}$ for all $n$. Since by assumption $\san$ is bounded, this implies that there exists $\delta >0$ such that  $\abs{z-\frac{\pn}{\qn}} \geq \delta'/\abs{\qn}^2$ for all $n$; thus the desired statement holds for $p/q$ of the form $\pn/\qn$ for some $n$; it may be noted that this does not involve the condition on $\sdeln$ as in the hypothesis. We shall now deduce the general statement from this using Proposition~\ref{app-pr}. 

 As  $\limsup |\deln|<\nu(\Gamma)$, by hypothesis, there exists $n_0\in \N$ such that 
 $$\nu (\Gamma)\inv\sup_{n\geq n_0}|\deln| <1<\inf |\zn|.$$ 
 Hence there exists $\lambda \in (0,1)$ such that $|\zn|>\lambda\inv$ and $|\deln|<\lambda\, 
 \nu (\Gamma)$ for all $n\geq n_0$. Then for $n\geq n_0$, $\theta_n$ as in  Proposition~\ref{app-pr} is less than $\lambda$ and hence by the proposition we have $|qz-p|\geq (1-\lambda) |\qn z-p_n|$, for all  $p,q\in \Gamma$, $q\neq 0$ such that $|\qnm|<|q|\leq |\qn|$, with  $n\geq n_0$. Thus for these $p,q$, 
 setting $\beta=(1-\lambda)\inv$ we get $|\qn z-\pn|\leq \beta |qz-p|$, and hence 
 $$|\qn|\,|\qn z-\pn|\leq \beta |qz-p|\,|\qn|
 \leq \beta |qz-p|\,|q|\,|\qn/\qnm|. $$  
 From the recurrence relations for $\sqn$, and the monotonicity assumption on $\set{\abs{\qn}}$ we have $|\qn/\qnm|\leq |\an|+1$ for all $n$. Hence 
 for all $n\geq n_0$, we get $|q|\,|qz-p|\geq |\qn|\,|\qn z-\pn|/\beta (|\an|+1)$. 
 As noted above there exists $\delta'>0$ such that $|z-\frac{\pn}{\qn}|\geq\delta'/|\qn|^2$ for all $n$, and hence the preceding conclusion implies that for all $q$ such that 
 $|\qnm|<|q|\leq |\qn|$ for some $n\geq n_0$ we have $|q|\,|qz-p|\geq \delta'/\beta (|\an|+1)=\delta''$, say, and so  $|z-\frac{p}{q}|\geq \frac{\delta''} {|q|^2}$. Since there are only finitely many pairs $(p,q)$ such that $|q|<|q_{n_0}|$ and 
 $|z-\frac{p}{q}|\leq \frac{\delta''}{|q|^2}$, we get that there exists $\delta>0$ such that 
 $|z-\frac{p}{q}|\geq \frac{\delta}{|q|^2}$, for all $p, q\in \Gamma$, $q\neq 0$. Thus   
 $z$ is badly approximable with respect to $\Gamma$.  
\end{proof} 

In the light of the comments in Remark~\ref{rem:nearest} it follows  that the characterization as in Corollary~\ref{badapp} holds in particular for expansions corresponding to the nearest integer algorithms on the Euclidean subrings, as established earlier in \cite{Hin}. In the next section we discuss examples of certain other algorithmic expansions for which also it holds.

\subsection{Circles of badly approximable numbers}

It was noted in \cite{Hin} that the case of Corollary~\ref{hbounded} for the nearest integer algorithm proved there implies existence of circles in $\C$ all whose elements are badly approximable with respect to $\Gamma$. We include here a description of such circles arising from the Corollary.  For any  $z \in \C$ and $r>0$ we shall denote by $C(z ,r)$ the circle with centre at $z$ and radius $r$, viz. $C(z,r)=\set{z \in \C \mid |z-\zeta|^2=r^2}$. 

\begin{corollary}\label{badcircles}
 Let $\Gamma$ be a Euclidean subring of $\C$ and $K$ be the quotient field of $\Gamma$. Let $\kappa \in K$ and $r>0$ such that $r^2\in \Q$. Then the following statements are equivalent. 

 i) All points of $C(\kappa, r)$ are badly approximable with respect to $\Gamma$. 

 ii) All points of $C(0, r)$ are badly approximable with respect to $\Gamma$. 

 iii) $r^2$ is not of the form $s/t$, where $s$ and $t$ are norms of some elements of $\Gamma$. 
\end{corollary}

\begin{proof}
 We note that for any $\zeta \in K$ the circle $C(\zeta, r)$, where $r$ is as in the hypothesis, is the set of zeroes of a Hermitian polynomial with coefficients in $\Gamma$ and hence by Corollary~\ref{badapp} and Corollary~\ref{hbounded} all the points of $C(\zeta, r)$ are badly approximable with respect to $\Gamma$ whenever $C(\zeta, r)\cap K=\emptyset$. Equivalence of (i) and (ii) now follows from the fact that $C(\zeta, r)$ contains an element  of $K$ if and only if $C(0,r)$ does. We note that if $p/q$  is an element of $K$ belonging to $C(0,r)$, where $p,q$ are coprime elements in $\Gamma$, and $r^2=s/t$, where $s,t$ are coprime (natural) integers, then we have $p\bar{p} t=q\bar{q} s$, and using that  $\Gamma$ is a unique factorization domain we get that $s=p\bar{p}$ and $t=q\bar{q}$; thus $s$ and $t$  are norms of the elements $p$ and $q$ respectively. Conversely if $r^2$ is of the form $s/t$ where $s$ and $t$ are norms of elements, say $p$ and $q$ respectively, of $\Gamma$, then $p/q\in C(0,r)\cap K$. This proves the equivalence of (ii) and (iii). 
\end{proof}

\begin{remark}
Choosing $n\in \N$ such that $\ds n\equiv 7 \,(\hbox{\rm mod 8})$ which is not a square modulo $3, 7$ and $11$, e.g. $n=1847=(8\times 3 \times 7 \times 11) - 1$, we get a circle defined by $|z|^2=n$, all whose points are badly approximable with respect to all the Euclidean subrings $\Gamma$ as above.  
\end{remark}

A number is said to be {\it well-approximable}, with respect to $\Gamma$, if it is not badly approximable with respect to $\Gamma$. Corollary~\ref{badcircles} in particular implies the following. 

\begin{corollary}
 Let $\Gamma$ be a Euclidean subring of $\C$. Then the set of complex numbers which are  well-approximable with respect to $\Gamma$ is totally disconnected. 
\end{corollary}

\begin{proof}
 The proof is immediate since given any distinct  elements, say $z$ and $w$, well-approximable with respect to $\Gamma$, there exists a circle $C(\kappa, r)$, with $\kappa \in K$ and $r>0$ and $r^2\in \Q$ such that all points of $C(\kappa, r)$ are badly approximable, and $z$ and $w$ are contained in the two disjoint components of $\C\backslash C(\kappa, r)$. 
\end{proof}

\section{Monotonicity of the denominator sequences}

In this section we discuss the monotonicity condition on the absolute values of the denominators of the convergents, involved in various results in the earlier sections, and in particular give new examples of continued fraction expansions for which the property holds. {Also, rather than limiting ourselves to considering sequences arising as continued fraction expansions of elements of $\C'$, we shall consider more general sequences of elements from the Euclidean rings and explore their properties akin to those of continued fraction sequences, and also the issue as to when they are continued fraction expansions of some $z\in \C'$.}\\ 

Let $\Gamma$ be a discrete subring of $\C$, and let $\san_{n=0}^\infty$ be a sequence in 
$\Gamma \backslash \{0\}$. Then the recurrence relations introduced in $\S\,2$ (which depend only on $\san$, not on the $z$ that was involved in the discussion there) yield sequences 
$\spn_{n=-1}^\infty$ and $\sqn_{n=-1}^\infty$ which we call the $\Qp$-pair corresponding to $\san_{n=0}^\infty$; it needs to be borne in mind however that various properties of the pair of sequences in \S\,2 are no longer guaranteed, and in particular, a priori, $\qn$ could be $0$ for some $n$, in the generality as above; however, when monotonicity of $\abs{\qn}$ is proved under certain conditions, they being nonzero would be an automatic consequence.  

We begin by fixing some notation. We denote by $B(z,r)$ the open disc centered at $z\in\C$ with radius $r>0.$ Also, $B(z,1)$ will be denoted simply by $B(z)$. 
For any subset $E$ of $\C$ we denote by $E\inv$ the subset $\set{z\in\C \mid z\inv\in E}$, and by $E^c$ the complement of $E$ in $\C$.

\begin{remark}\label{circinv}
Many computations in the sequel involve $B(z,r)\inv$, where $z\in \C$ and $r>0$. We note that when $|z|>r$, 
$$B(z,r)\inv=B\left(\frac{\zb}{|z^2|-r^2}, \frac{r}{|z^2|-r^2}\right),$$
as can be verified directly. In general the set of inverses of $B(z,r)$ can be determined by identifying its boundary circle, passing through the inverses of three suitably chosen points on the boundary of $B(z,r)$, and noting that if $0\in B(z,r)$, then $B(z,r)\inv$ is the complement of the closed disc bounded by the circle $C$, and is the interior of the disc if $0\notin B(z,r)$. Clearly, similar remarks apply to $\Bb(z,r)$, the closure of $B(z,r)$. 
For convenience we shall in general skip the details of the computations made along these lines.   
\end{remark}

\subsection{General conditions ensuring monotonicity of $\{\abs{\qn}\}$}

Let the notation be as above and for any $n\geq 0$ let $\rn=\qn/\qnm$. Then $\set{\rn}$ satisfy the recurrence relations $\rn=\an+\rnm\inv$ for all $n\geq 1$, and the object of our exploration concerns whether $|\rn|\geq 1$ or, in a stronger form (for strict monotonicity), $|\rn|>1$. With regard to this question we note the following.

\begin{lemma}\label{triple}
 Let $\gamma_0, \gamma_1, \gamma_2 \in \C$ be such that $|\gamma_j|>1$ for $j=0$ and $1$, $0<|\gamma_2|\leq 1$, and $\gamma_j-\gamma_{j-1}\inv\in \Gamma$ for $j=1,2$. For $j=1,2$ let  $\alpha_j=\gamma_j-\gamma_{j-1}\inv $ and suppose that $|\alpha_2|>1$. Then $|\alpha_2|< 2$, $\gamma_1\in \Bb\left(\frac{-\overline{\alpha_2}}{|\alpha_2|^2-1}, \frac{1}{|\alpha_2|^2-1}\right)$, and $\alpha_1\in B\left(\frac{-\overline{\alpha_2}}{|\alpha_2|^2-1}, \frac{|\alpha_2|^2}{|\alpha_2|^2-1}\right)$, viz. $\abs{(|\alpha_2|^2-1)\alpha_1+\overline{\alpha}_2}< |\alpha_2|^2$. 
\end{lemma}

\begin{proof}
 Clearly $|\alpha_2| \leq |\gamma_2|+|\gamma_1\inv|<2$, as $|\gamma_2|\leq 1 <|\gamma_1|\inv$. Now, we have $\gamma_1\inv=\gamma_2-\alpha_2 \in \Bb(-\alpha_2, |\gamma_2|)\subset \Bb(-\alpha_2, 1)$. 
 Since by  assumption $|\alpha_2|>1$, Remark~\ref{circinv} now implies that $\gamma_1\in \Bb\left(\frac{-\overline{\alpha_2}}{|\alpha_2|^2-1}, \frac{1}{|\alpha_2|^2-1}\right)$.
 As $\alpha_1=\gamma_1-\gamma_0\inv$ and $|\gamma_0|>1$ the preceding conclusion implies also that $\alpha_1$ is contained in the open ball $B\left(\frac{-\overline{\alpha_2}}{|\alpha_2|^2-1}, \frac{1}{|\alpha_2|^2-1}+1\right)
 =B\left(\frac{-\overline{\alpha_2}}{|\alpha_2|^2-1}, \frac{|\alpha_2|^2}{|\alpha_2|^2-1}\right)$, which proves the  lemma. 
\end{proof}

\begin{remark}
The conditions $|\alpha_1|>1 $, $1<|\alpha_2|<2$ and $\abs{(|\alpha_2|^2-1)\alpha_1+\overline{\alpha_2}}<|\alpha_2|^2$
can be interpreted  for the rings $\Gamma=\G$ and $\E$, of  Gaussian and Eisenstein integers, respectively, as follows (though a similar description can be given in the case of other Euclidean rings, we shall not go into it here, as the expressions get cumbersome, and the corresponding observations will not be involved in our discussions in the sequel).:

i) Let $\Gamma=\G$. Then $\alpha_2=(1+i)\,i^l$ for some $l$ (where $0\leq l\leq 3$). 
Writing  $\alpha_1 i^{l}$ as $x+iy$, where $x,y \in \Z$, 
the conditions on $\alpha_1$ are given by $|\alpha_1|>1$ and $\abs{\alpha_1+(1-i)i\inv}<2$, 
and translate to the conditions $x^2+y^2>1$ and $(x+1)^2+(y-1)^2<4$. There are $6$  solutions for the pair $(x,y)$, given by $-x$ and $y$ between $0$ and $2$ such that $x+y>1$. Thus $\alpha_1=\beta\, i\inv$ where $\beta \in \set{-x+iy \mid -2\leq x\leq 0,\, 0\leq y \leq 2,\, -x+y>1}$.
   
ii) Let $\Gamma =\E$.  Then $\alpha_2=j\rho^l$ for some $l$, $0\leq l \leq 5$, where $j=\sqrt{3}i$ and $\rho =\frac{1}{2}(1+\sqrt{3}i)$. Then the conditions on $\alpha_1$ are  that $|\alpha_1|>1$ and  $|2\alpha_1 -j\rho\inv|<3$, or equivalently $|2\alpha_1 \rho\inv -j|<3$. Writing  $\alpha_1 \rho^{l}$ as $\frac{1}{2} (x+jy)$, where $x,y \in \Z$ and $x+y\in 2\Z$, we get the conditions $x^2+3y^2>4$ and $x^2+3(y-1)^2<9$. The only solutions for $(x,y)$  are $(0,2)$ and $(\pm 2,2)$. Thus $\alpha_1=\beta  \rho\inv$ where $\beta \in \set{j, \pm 1+j}$.
\end{remark}

The Lemma yields the following criterion which was proved earlier in \cite{D-gen}; in the proof in \cite{D-gen} justification at the starting point is missing in the inductive argument, which would be cleared by the following argument. 

\begin{corollary}\label{cor:mon}
 Let the notations $\Gamma$, $\san$, and $\sqn$  be as above. 
 Suppose that $|\an|>1$ for all $n\geq 1$. Let $\rn=\qn/\qnm$ for all $n\geq 1$. Then at least one of the following holds:

 i) $|\qnp|>|\qn|$ for all $n\geq 0$. 

 ii) There exists $m\geq 1$ such that $|\rn|>1$ for $n=1,\dots, m$, $|r_{m+1}|\leq 1$, 
 $|a_{m+1}|< 2$, $$r_m\in \Bb\left(\frac{-\bar a_{m+1}}{|a_{m+1}|^2-1}, \frac{1}{|a_{m+1}|^2-1}\right)  \hbox{ and } a_m\in B\left(\frac{-\bar a_{m+1}}{|a_{m+1}|^2-1}, \frac{|a_{m+1}|^2}{|a_{m+1}|^2-1}\right).$$
\end{corollary}

\begin{proof}
 Since $q_0=1$ and $q_1=a_1$, by hypothesis we have $|q_1/q_0|=|a_1|>1$. Now suppose $|q_2|\leq |q_1|$. We show that in this case statement (ii) holds for $m=1$. We have $r_2=|q_2/q_1|\leq 1$  and, since $|q_2/q_1|= |a_2+q_1\inv|=|a_2+a_1\inv|$,
 $|a_2| \leq |q_2/q_1|+|a_1|\inv\leq 1+|a_1|\inv<2$. Also as $|a_2+a_1\inv|=|q_2/q_1|\leq 1$, $a_1\inv\in \Bb(-a_2, 1)$, and hence by Remark~\ref{circinv}  $a_1\in \Bb\left(\frac{-\overline{a_2}}{|a_2|^2-1}, \frac{1}{|a_2|^2-1}\right)$. 
 Since $r_1=a_1$, this shows that statement (ii) holds for $m=1$.
  
 Now suppose that $|\qnp|>|\qn|$ for $n=0, \dots, m-1$ and $|q_{m+1}|\leq |q_m|$, for some $m\geq 2$. Now for $j=0,1$ and $2$ define  $\gamma_j=q_{m+j-1}/q_{m+j-2}$.  We see that $|\gamma_j|>1$ for $j=0,1$, $|\gamma_2|\leq 1$, and for $j=1,2$, $\gamma_j-\gamma_{j-1}\inv= a_{m+j-1}\in \Gamma$. Thus the conditions in Lemma~\ref{triple} are satisfied with $\alpha_1=a_m$ and $\alpha_2=a_{m+1}$.
 The lemma then implies that statement~(ii) of the corollary holds. 
\end{proof}
 
The following Corollary which readily follows from Corollary~\ref{cor:mon} shows the flexibility available in principle for modifying continued fraction expansions without losing the monotonicity property for the denominators of the convergents.  
 
\begin{corollary}\label{cor:gen}
 Let $\set{F_\gamma}$ be a family of subsets of $B(0)$ such that 

 i) $\C= \bigcup_{\gamma \in \Gamma} \gamma +F_\gamma$.

 ii) For $\theta \in \Gamma$ with $|\theta|=1$, $\theta + F_\theta \subset \C \backslash (\bigcup_{\gamma \in \Gamma} F_\gamma\inv)$.

 iii) For $\theta \in \Gamma$ with $1<|\theta|<2$ and $\varphi \in B\left(\frac{-\overline{\theta}}{|\theta|^2-1}, \frac{|\theta|^2}{|\theta|^2-1}\right)$, $F_\varphi\inv\cap (\theta +F_\theta)= \emptyset$. 

 Let $z\in \C'$, $\san$ be a continued fraction expansion of $z$, $\szn$  the corresponding iteration sequence, and $\sqn$ the sequence of denominators from the corresponding $\Qp$-pair. Suppose that $z_n\in \an+F_{\an}$ for all $n\geq 0$. Then $|\qnp|>|\qn|$ for all $n$. 
\end{corollary}

We note that all elements involved in condition~(ii) of Corollary~\ref{cor:mon} are of absolute value less than $3$. Hence if $f:\C\to\Gamma$ and $g:\C\to\Gamma$ are two algorithms such that $f\inv(a)=g\inv(a) $ for all $a\in \Gamma$ with $|a|< 3$, then the condition holds for $f$ if and only if it holds for $g$, irrespective of their values at other points. Hence 
given an algorithm $f:\C\to \Gamma$ such that for the corresponding continued fraction expansions $\san$ we have $|\an|>1$ for all $n\geq 1$ and condition~(ii) of Corollary~\ref{cor:mon} does not hold, we can construct modified algorithms for which also the monotonicity statement holds, provided for the new algorithm we ensure the condition $|\an|>1$ for all $n\geq 1$. This is illustrated by the following example. 
 
\begin{example}\label{Eisen-exampl}
Let $\Gamma=\E$, the ring of Eisenstein integers. Let $H$ be the closed region whose boundary is the hexagon with vertices at $\rho^l/\sqrt{3}$, $l=0,\dots, 5$, where $\rho=\frac{1}{2}(1+i\sqrt{3})$, a $6$th root of unity. The nearest integer algorithm for $\E$, say $f$, corresponds, upto ambiguities at the boundaries, to setting $f(z)=a$ if $z\in a+H$. Now consider another algorithm $g$ such that if $g(z)=a$ and $|a|\leq 2$ then $z\in a+H$ and, in general, if $g(z)=a$ then $g(z)\in a+\frac{5}{4} H$ (there would be multiple possibilities for points away from the origin,  and the specific algorithm has to pick one of them, by a  convention).  As observed above to prove monotonicity of $\abs{\qn}$'s corresponding to the continued fraction expansion $\san$ of a $z\in \C'$ it suffices to show that $|\an|>1$ for all $n\geq 1$. By symmetry considerations it suffices to show that $\an\neq 1$ for any $n\geq 1$. 

The set $\set{\zeta -g(\zeta) \mid \zeta \in \C}$ is by construction contained in $\frac{5}{4} H$. The set of inverses of the latter set may be seen to be the complement of the $6$ discs $\rho^l B(\frac{4}{5}, \frac{4}{5})$, $l=0,\dots ,5$. It can also be verified directly that $1+H$ is contained in $B(\frac{4}{5}, \frac{4}{5})$. Therefore $\an\neq 1$, and hence, as argued above, we get that $|\qnp|>|\qn|$ for all $n\geq 0$. We may mention here that the factor $5/4$ chosen above is not the maximum possible; determination of the optimal value involves some cumbersome algebra which it does not seem worthwhile going into at this stage, so we have chosen an ad hoc value known  from our calculations to be close to the optimal one. 

If in place of the factor $\frac{5}{4}$ we restrict to algorithms $g$ for which $g(z)\in a +\chi H$ where $1<\chi <\frac{3}{2}(\sqrt{3}-1)$, then for the expansions corresponding to the algorithm we get $|\zn|>1 +\sqrt{3}$, so the condition as in Remark~\ref{rem:badapp} is satisfied in this case, and hence for these  expansions, by Proposition~\ref{badapp}, boundedness of partial quotients characterizes the number being badly approximable with respect to $E$. 
\end{example}

\begin{remark}
One of the questions that one might consider in this respect is to have a common set $F$ as $F_\gamma$ for all $\gamma \in \Gamma$. It was shown in \cite{D-gen} that when $\Gamma=\E$, the ring of Eisenstein integers, the conditions as in Corollary~\ref{cor:gen} are satisfied with $F_\gamma =F=B(0,r)$ for all $r \in \left(\frac{1}{\sqrt{3}}, \sqrt{\frac{1}{4}(5-\sqrt{13}})\right)$. 
\end{remark}

\section{More on monotonicity for expansions in Gaussian integers}

In this section we shall dwell further on the theme of monotonicity of the denominators, restricting to the case of the ring $\G=\Z[i]$ of Gaussian integers. General continued fractions were studied extensively in this case in \cite{DN}, establishing results on monotonicity of the denominator sequences, and discussing some applications. We describe some classes of examples in continuation of the results there.

We shall denote by $\Sigma$ the group of transformations of $\C$ generated by the transformations $z\mapsto -z$ and $z\mapsto \zb$. We define $\sigma_y\in \Sigma$ to be the transformation defined by $\sigma_y(z)=-\zb$ for all $z\in \C$, viz. the reflection in the $y$ axis. 
 
We begin by recalling the following definition from \cite{DN}. 
 
\begin{definition}
A sequence $\san_{n=0}^\infty$ in $\G$  is said to satisfy {\it Condition\,(H)} if $|\an|>1$ for all $n\geq 1$, and for all $k,l$ such that $1 \leq k <l$, $|a_l|=\sqrt{2}$ and $a_j=2\sigma_y^{l-j}(a_l)$ for all $j=k+1, \dots l-1$ (none if $k=l-1$), we have either $a_k=2\sigma_y^{l-k}(a_l)$ or $\abs{a_k-\sigma_y^{l-k}(a_l)}\geq 2$. 
\end{definition} 
 
We recall that by  Proposition~5.9 of \cite{DN} if $\san$ is a sequence in $\G$ satisfying Condition~(H) and $\spn, \sqn$ is the corresponding $\Qp$ pair, then $|\qn|>|\qnm|$ for all $n\geq 1$.

\subsection{Variations on the Hurwitz algorithm} 

In this subsection we shall show that for algorithms obtained by certain  modifications of the Hurwitz algorithm (see Example~\ref{h-perturb} below)  monotonicity of the denominator sequence holds. 

For convenience we shall consider continued fraction algorithms to be defined only over $\C'$; the extension to $\C$ can be arbitrary, but will not play a role. For a continued fraction algorithm $f:\C'\to\G$ we call the set $\set{z-f(z) \mid z\in \C'}$ the {\it fundamental set} of $f$. 

\begin{proposition}\label{geom-hurwitz}
 Let $f:\C' \to \G$ be an algorithm for continued fraction expansions and $F$ be the fundamental set of $f$. For $a\in \frak G$ let $C(a)=f\inv(\set{a}) \cap F\inv$. Suppose that $C(a)=\emptyset$ when $|a|=1$ and that there exists a subset $Q$ of $F$  such that the following conditions are satisfied for any $a=\sigma (1+i)$, $\sigma \in \Sigma$: 

 i) $C(a)\inv$ is contained in $\sigma (Q)$;

 ii)  if $b\in B(\sigma_y( a),2)\cap \G$, and $b\neq -2\bar{a}$ then 
 $(b+\sigma (Q))\cap F\inv=\emptyset$;

 iii) for $b=-2\bar{a}$, $(b+\sigma (Q))\inv$ is contained in $\sigma_y (Q)$. 

\noindent Let $z\in \C'$, $\san$ be the continued fraction expansion of $z$ with respect to $f$ and $\sqn$ be the corresponding sequence of denominators. Then $|\qn|<|\qnp|$ for all $n\geq 0$. 
\end{proposition}

\begin{proof}
 We shall show that $\san$ satisfies Condition~(H); as noted above, by Proposition~5.9 of \cite{DN} this implies the desired monotonicity assertion. 
 We note that since by hypothesis $C(a)=\emptyset$ for all $a\in\G$ with $|a|=1$, we have $|\an|>1$ for all $n\geq 1$, and hence it remains only to verify the second part of the condition in the definition. 
 Let $k,l$ such that $1\leq k <l$, $|a_l|=\sqrt{2}$ and $a_j=2\sigma_y^{l-j}(a_l)$ for all $j=k+1, \dots l-1$ (none if $k=l-1$), be given. 
 As $1<|a_l|<2$, we have $a_l=\pm 1\pm i$.  
 For convenience we shall consider only $a_{l}=1+i$; the other cases follow from symmetry, or alternatively can be proved by an analogous argument. 

 Now let $\szn$ be the iteration sequence corresponding to the  expansion. 
 Suppose that $|a_{l-1}-\sigma_y ( a_l)|<2$.  For $a_l=1+i$ this implies that $a_{l-1}$ is one of $-2, -1+i, 2i, -2+i, -1+2i$ or $-2+2i$. 
 Now,  $z_{l}\in C(a_l) =C(1+i)$, and hence $z_{l-1}-a_{l-1}=z_l\inv\in Q$, by Condition~(i) in the hypothesis. Thus $z_{l-1}\in a_{l-1}+Q$ and since 
 ${l-1}\geq 1$, we also have $z_{l-1}\in F\inv$. Hence Condition~(ii) implies that $a_{l-1}=-2+2i$. In particular this proves the desired assertion in the case when $k=l-1$. 

 Now, $z_{l-1}\in a_{l-1}+Q=-2+2i+Q$ and hence by Condition~(iii) 
 $z_{l-1}\inv\in\sigma_y(Q)$. Thus $z_{l-2}- a_{l-2}\in \sigma_y(Q)$. Suppose that 
 $\abs{a_{l-2}-a_l}<2$; as $a_l=1+i$ this means that $a_{l-2}$ is one of 
 $2, 1+i, 2i, 2+i, 1+2i$ or $2+2i$. We have 
 $z_{l-2}\in a_{l-2}+\sigma_y(Q)$ and, as $1\leq k\leq l-2$, $z_{l-2} \in F\inv$. Hence by Condition~(ii), applied with $a=\sigma_y(1+i)=-1+i$, we get that the only possibility is $a_{l-2}=2+2i$. By repeated application of this argument successively until reaching $k$ we see that either $a_k=2\sigma_y^{l-k}(a_l)$ or $\abs{a_k-\sigma_y^{l-k}(a_l)}\geq 2$. This shows that $\set{a_k}$ satisfies Condition~(H), and thus proves the proposition. 
\end{proof}

It can be seen that the conditions of Proposition~\ref{geom-hurwitz} hold for continued fraction expansions with respect to the nearest integer algorithm for $Q=Q_h$ where $$Q_h=\set{z=x+iy \mid x\in [0,\mfrac{1}{2}],\, y\in [-\mfrac{1}{2},0] \hbox { \rm and }(1-x)^2+(1+y)^2< 1}.$$ 
 We now describe a somewhat more general class of algorithms for which also the  conditions of the proposition hold for a modified choice of $Q$.

\begin{example}\label{h-perturb}
	Let $S=\set{z=x+iy \mid |x| \leq \frac{1}{2}, |y|\leq \frac{1}{2}}$ and for $0<r<1-\frac{1}{\sqrt{2}}$ let 
	$S_r=S\cup (\cup_{\sigma \in\Sigma} B(\frac 12 \sigma (1+i), r)$; then $S_r$ is contained in the unit disc. 
	Let $f:\C'\to \G$ be any continued fraction algorithm such that 
	$f\inv(a)\subset a+S_r$  for all $a\in\G$. We shall show that for small $r>0$, the conditions as Proposition 6.2 are satisfied for the continued fraction expansions of any $z\in\C'$ with respect to $f$.
	
	By hypothesis the fundamental set $F$ of $f$  is contained in $S_r$. We note that for all $\sigma \in \Sigma$,  $B(\frac{1}{2}\sigma (1+i), r)\inv$ is contained in $B(\overline{\sigma} (1+i),  \frac{2r}{1-r\sqrt{2}})$.  Hence $F\inv$ is contained in the set
	$$\Phi:=\Big[\bigcap_{\sigma\in\Sigma} B(\sigma(1))^c\Big] \bigcup \Big[\bigcup_{\sigma \in\Sigma} B(\sigma(1+i),r')\Big],$$
	where $r':=2r/(1-r\sqrt 2)$.
 	Now let $r>0$ be such that $r+r'<\frac{1}{\sqrt{2}}$ and $r'<(\sqrt{2}-1)$; 
	former condition may be seen to be equivalent to $r<\sqrt 2-\sqrt {3/2} \approx 0.1894\dots$; the condition however involves an ad hoc choice made for simplicity of computation and the value does not represent an optimal one for the class of examples. 
	
	We choose $Q_r=Q_h\cup B(\frac{1}{2}(1-i), r)$, where $Q_h$ is as defined above. Observe that $Q_r=S_r\cap B(1-i).$ We show that the conditions of Proposition~\ref{geom-hurwitz} are satisfied for $Q=Q_r$. 
	We note that since for $a=1$, $a+S_r$ is contained in $B(1)$ and using that $r+r'<\frac{1}{\sqrt{2}}$ it can be seen that it does not intersect $B(\sigma (1+i), r')$ for any $\sigma \in \Sigma$. This shows that $C(1)=\emptyset$ and similarly we get that $C(a)=\emptyset$ for all $a\in \G$ with $|a|=1$. 
	We next verify conditions~(i), (ii) and (iii) of Proposition\ref{geom-hurwitz}.
	We carry out the verification  for $a=1+i$, viz. when $\sigma$ is the identity transformation; for the other cases the conclusion follows from symmetry (or by analogous argument). 
	
	We note that $C(1+i)$ is contained in $B(1+i)$ and hence $C(1+i)\inv$ is contained in $B(1+i)\inv=B(1-i)$. Since $S_r\cap B(1-i)=Q_r$, this shows that $C(1+i)\inv$ is contained in $Q_r$, thus proving Condition~(i), for the case at hand. 
	Next consider Condition~(ii). We have to show that for $b\in\set{-2, -1+i, 2i, -2+i, -1+2i}$, $(b+Q_r)\cap F^{-1}=\emptyset$. So it would suffice to check that $b+Q_r$ does not intersect the set $\Phi$ as above. For all the values of $b$ as above it is immediate that $b+Q_h$ is contained in either $B(-1)$ or $B(i)$, and moreover the condition that $r+r'<\frac{1}{\sqrt{2}}$ and $r'<(\sqrt{2}-1)$  implies that it does not intersect  $B(\sigma (1+i),r')$  for any $\sigma \in \Sigma$. Similarly we see that $b+B(\frac{1}{2}(1-i), r)$ is contained in one of 
	$B(-1)$ and $B(i)$ and on account of the condition that $r+r'<\frac{1}{\sqrt{2}}$ it does not intersect $B(\sigma (1+i),r')$  for any $\sigma \in \Sigma$. Thus we get that $(b+Q_r)\cap F\inv=\emptyset$ as desired, confirming condition (ii). 
	
	To verify Condition~(iii) for $a=1+i$ we have to show that $((-2+2i)+Q_r)\inv$ is contained in $\sigma_y(Q)$. We in fact show that it is contained in $\sigma_y(Q_h)=S\cap B(-1-i).$ Since $Q_h\subset S$, we have 
	$Q_r\subset S \cup B(\frac{1}{2}(1-i),r)$, and thus $(-2+2i)+Q_r$ is contained in 
	$\left[(-2+2i)+S\right]\cup B(\frac{3}{2}(-1+i),r)$ which does not intersect any $B(\sigma(1))$ for any $\sigma\in\Sigma$, since $r+r'<\frac{1}{\sqrt{2}}.$ This shows that $[(-2+2i)+Q_r]\inv \subset S$ since $S\inv=\cap_{\sigma\in\Sigma} B(\sigma(1))^c$. On the other hand $(-2+2i)+Q_r$ is also contained in $B(-1+i)$ and hence its inverse is contained in $B(-1+i)\inv=B(-1-i)$. Hence $\left[(-2+2i)+Q_r\right]\inv\subset S\cap B(-1-i)= \sigma_y (Q_h)$, which proves  Condition~(iii) for $a=1+i$. 
	As noted above, by symmetry considerations this shows that the conditions of  Proposition~\ref{geom-hurwitz} are satisfied for the algorithms as above. Hence by the proposition $|\qn|<|\qnp|$ for all $n$, for the corresponding denominator sequence $\sqn$.  
\end{example}

\begin{remark}
In place of the discs $B(\sigma (1+i), r)$ added at the corners of $S$ in this example, one could also consider adding other shapes, and in certain cases the same conclusion as above could be arrived at, without actually the shape being contained in an $S_r$ for which the result is proved. 
In \cite{DN} an algorithm was introduced, by the name {\it PPOI algorithm} (an acronym of ``partially preferring odd integers"), as an illustration of an algorithm with  the property that the set of the corresponding continued fraction expansions  of numbers from $\C\backslash K$ can be characterised as the set of all sequences that preclude occurrence of (finite) blocks from a finite collection (this property does not hold for the Hurwitz algorithm); we shall not go into further details of this aspect here. 
That algorithm also involved adding a certain shape at each corner, but the latter is not contained in any disc of radius less than $\sqrt 2/6$.  The monotonicity assertion as above is however upheld there, by a somewhat intricate method.  
\end{remark}

\subsection{Algorithms with a noncompact fundamental set}

The examples in \cite{DN}, for which the monotonicity of the denominator sequence is upheld have their fundamental sets contained in a ball $B(r)$ with $r<1$. In this subsection we shall discuss some examples for which this is not the case. 
 
Let $H$ denote the square $\set{z=x+iy \mid |x|+|y|\leq 1}$, the closed square with vertices at $\pm 1$ and $\pm i$. It is straightforward to see that 	
$$H\inv=\bigcap_{\sigma\in\Sigma} B\left(\sigma\left(\mfrac{1}{2}(1+i)\right) ,\mfrac{1}{\sqrt{2}} \right)^c.$$

A Gaussian integer $z=x+iy$ is said to be \emph{even} or \emph{odd}, respectively, according to whether the sum $x+y$ is even or odd. We prove the following. 

\begin{theorem}\label{Gaussian}
 Let $z\in \C'$, $\san$ be a continued fraction expansion of $z$ 	
 and $\szn$ be the corresponding iteration sequence. Suppose that for all $n\geq 1$ the following conditions are satisfied:
 i) $\zn\in \an+ H$, and 
 ii) if $|\an|<3$ then $\an$ is even. 
 Then for the corresponding denominator sequence $\sqn$ from the $\Qp$-pair we have $|\qnp|>|\qn|$ for all $n\geq 0$. 
\end{theorem}

\begin{proof}
 We first show that under the conditions as in the hypothesis, if  $\an=\sigma(1+i)$, for some $n\geq 1$ and $\sigma \in \Sigma$, then $\anp\in \sigma (L)$, where $L$ is  the affine semi-plane $\set{x+iy \mid y\leq x+1}$. Since $H$ is $\Sigma$-invariant, it suffices to prove this in the case when $\sigma$ is the identity transformation. Now let $\an=1+i$. As noted earlier,
 $H\inv\subset \Bb(\frac{1}{2}(1+i),\frac{1}{\sqrt{2}})^c\cup \{\pm 1, \pm i\},$
 and hence
 $(H\inv-\an)\cap\C'\subset \Bb\big(-\frac{1}{2}(1+i),\frac{1}{\sqrt{2}} \big)^c.$ 
 It may be verified that $\Bb\big(-\frac{1}{2}(1+i),\frac{1}{\sqrt{2}} \big)\inv =\set{x+iy \mid y\geq x+1}$. 
 Hence $(H\inv-\an)\cap\C'\subset L^{\circ}$, where 
 $L^{\circ} =\set{x+iy \mid y< x+1}$ is the interior of $L$ as in the hypothesis. Since $\znp\inv=\zn-\an=(\znm-\anm)\inv-\an$ (as $n\geq 1$), and further, since $\znp \in \C'$, it follows that $\znp \in L^\circ $. It may be observed that for $a\in \Gamma$, $(a+H)\cap L^\circ$ is nonempty only if $a\in L$. Since $\znp \in (\anp +H)$, the preceding conclusion implies   that $\anp \in L$. 

 Now suppose that the monotonicity assertion in the theorem does not hold. 
 Noting that for $a\in\G$ the condition $1<|a|<2$ implies that $a=\sigma (1+i)$, for some $\sigma \in \Sigma$, we get from  Corollary~\ref{cor:mon} that there exists $m\geq 1$ such that $|\rn|>1$ for $n=1,\dots, m$, $a_{m+1}=\sigma (1+i)$, and $r_m\in B(\sigma(-1+i))$. 
 We show that $a_m\neq \sigma (-1+i)$.  Suppose if possible that $a_m= \sigma (-1+i)=\sigma\sigma_y  (1+i)$. Then  by the first part we would have $a_{m+1} \in \sigma \sigma_y (L)$. Since $a_{m+1}=\sigma (1+i)$, this implies that $1+i\in \sigma_y (L)$, namely $-1+i\in L$, which is a contradition, and hence $a_m\neq \sigma (-1+i)$. 

 We have thus shown that there exists $m\geq 1$ such that $|\rn|>1$ for $n=1,\dots, m$, and there exists $\nu=\pm 1\pm i$ such that $r_m \in B(\nu)$ and $a_m\neq \nu$. 
 We shall show that the statement implies that   
 $\rn-a_n\in R:=\cup_{\sigma \in \Sigma}B(\sigma (1+i)) $  for all $n=1,\dots, m$. Since $r_1=a_1$ and $0\notin R$ this would give a contradiction. 

 In view of the symmetry it suffices to consider the case $\nu =1+i$. 
 Thus we have $a_m\neq 1+i$, and $r_m\in B(1+i)$. Since $r_m\in  B(1+i)\cap B(a_m)$  we get that  $\abs{a_m-(1+i)}<2$, which further means $\abs{a_m-(1+i)}\leq \sqrt{2}$. Hence $|a_m|\leq 2\sqrt{2} <3$, and by the condition in the hypothesis it is even. 
 Since $a_m\neq 1+i$, and nonzero, this  implies that $a_m\in \set{2, 2i, 2+2i}$. We see from this that $(1+i)-a_m$ is of the form $\pm 1 \pm i$. Thus $(1+i)-a_m=\sigma (1+i)$, for some $\sigma \in \Sigma$. Since $r_m\in B(1+i)$ this shows that  $r_m-a_m \in B(\sigma (1+i)) \subset R$. In particular this completes the proof in the case $m=1$. We shall prove the general case using downward induction. For this it would suffice to show that under the conditions in the hypothesis there exists $\nu=\pm 1\pm i$ such that $r_{m-1}\in B(\nu)$ and $a_{m-1}\neq \nu$. 

 Recall that $r_{m-1}\inv =r_{m}-a_{m}$ is contained in $B(\sigma (1+i) )$ for the (unique) $\sigma \in \Sigma$ such that $\sigma (1+i)=1+i-a_m$. 
 Hence $r_{m-1}\in B(\sigma (1+i))\inv = B(\sigma (1-i))$. Let $\tau \in \Sigma$ be defined by $\tau (\zeta) =\overline {\sigma (\zeta)}$ for all $\zeta \in \C$.  Then  we have  $r_{m-1}\in B(\tau (1+i))$. Suppose, if possible, that $a_{m-1}=\tau (1+i)$. 
 Then by the first part of this proof $a_m\in\tau(L).$ We shall show using the possible choices for $\an$ (as noted earlier) that, that is not the case.
 If $a_m=2$ then $\sigma (x+iy)=-x+iy$, so $\tau (x+iy)=-x-iy$ for all $x,y\in \R$, and $a_m=2$ is not contained in $\tau (L)$. Similarly if $a_m=2i$ then $\sigma (x+iy)=x-iy$ and hence
 $\tau (x+iy)=x+iy$ for all $x,y$, so $\tau (L)=L$ and it does not contain $2i$, and if $a_m=2+2i$ then $\sigma (x+iy)=-x-iy$, so $\tau (x+iy)=-x+iy$ for all $x,y\in \R$ and $\tau (L) $ does not contain $2+2i$. Hence $a_{m-1}\neq \tau (1+i)$. This completes the inductive argument and we get that $\rn-\an\in R$ for all $n=1,\dots, m$. As noted above this is a contradiction since $r_1=a_1$ and $0\notin R$. Therefore $|\qnp|> |\qn|$ for all $n\geq 0$. 
\end{proof}

The following is an immediate consequence of Theorem~\ref{Gaussian}, Corollaries~\ref{Lag-mon} and~\ref{hbounded}.

\begin{corollary}\label{cor-even}
 Let the notation and hypothesis be as in Theorem~\ref{Gaussian}. Then  
 $z$ is a quadratic surd if and only if $\san$ is eventually periodic and it is  badly approximable with respect to $\G$ if and only $\san$ is bounded. 
\end{corollary}

\begin{example}
Let $f:\C \to \G$ be a continued fraction algorithm defined by taking
$f(z)$, $z\in \C$ to be the even integer nearest to $z$; when there are more than one at the same distance, a suitable convention may be followed for making the choice.  For the corresponding continued fraction expansions it can be seen that the fundamental set $\Phi_f=H$. As no odd integer appears among the partial quotients,   
by Theorem \ref{Gaussian} we get $|\qn|> |\qnm|$ for all $n\geq 1$. Thus, in particular, the conclusions as in Corollary~\ref{cor-even} hold in this case.  For more details about this algorithm see \cite{Hur-J}; this paper in particular contains a proof of the monotonicity of the absolute values of the denominators and the version of Lagrange theorem, but the details seem to be complicated, and unclear to the authors. 
Theorem~\ref{Gaussian} places this example in a broader perspective. 
\end{example}

\begin{example} 
Let $\Lambda$ be the subset of $\G$ consisting of elements of the form $x+iy$ where $x,y\in \Z$ and $|x+y|$ is either $0,2,4$ or an odd number which is at least $5$. It can be seen that $\Lambda +H=\C$ and hence we can have an algorithm $f:\C\to \Lambda$ such that $z-f(z)\in H$; e.g. we may take $f(z)$ to be the element of $\Lambda $ nearest to $z$, breaking ties through a suitable convention. We see that the conditions of Theorem~\ref{Gaussian} are satisfied in this case and hence we get that $|\qnp|>|\qn|$ for all $n$ in the notation as above.   
\end{example}

\vskip1cm

\begin{flushleft}
S.G. Dani  \hfill Ojas Sahasrabudhe\\
UM-DAE Centre for Excellence in \hfill Department of Mathematics\\
Basic  Sciences, University of Mumbai  \hfill Indian Institure of Technology Bombay\\
Santacruz, Mumbai 400098, India  \hfill Powai, Mumbai 400076, India\\
\end{flushleft}

\end{document}